\documentclass[12pt]{article}

\usepackage{pgf,tikz}
\usepackage{mathrsfs}
\usetikzlibrary{arrows}

\usepackage{graphicx} 
\graphicspath{{figures/}} 

\usepackage{fullpage, url,amsmath,amsfonts,amssymb,mathtools,mathrsfs,graphicx,  algorithm, float, sansmath,epstopdf,color,caption,enumitem,tabularx}
\usepackage[final]{pdfpages}
\usepackage{amsthm}
\usetikzlibrary{automata,topaths}
\usetikzlibrary{decorations.pathreplacing,shapes.misc}
\usepackage{fancyhdr}

\def\beq{\begin{equation}}
\def\eeq{\end{equation}}
\def\baq{\begin{eqnarray}}
\def\eaq{\end{eqnarray}}
\def\baqn{\begin{eqnarray*}}
\def\eaqn{\end{eqnarray*}}

\newcommand{\ball}{\mathbb{B}}

\usepackage{multirow}
\usetikzlibrary{calc,arrows}
\theoremstyle{plain}
\newtheorem{definition}{Definition}
\newtheorem{remark}{Remark}

\newtheorem{example}{Example}
\newtheorem{theorem}{Theorem}
\newtheorem{lemma}[theorem]{Lemma}

\usepackage{blindtext}

\usepackage[colorlinks,linkcolor=blue,citecolor=red]{hyperref}

\usepackage{xcolor, framed}

\newcommand{\R}{{\mathbb R}}

\newcommand{\interior}{{\rm int}\kern 0.06em}

\def\<{\langle}
\def\>{\rangle}

\newcommand\dom{{\rm dom}}%

\newcommand{\gra }{\operatorname{gra}}
\newcommand{\ran }{\operatorname{ran}}
\newcommand{\zer}{\operatorname{zer}}

\newcommand{\sgn}{\ensuremath{\operatorname{Sign}}}

\usepackage{ulem}

\usepackage{subcaption}

\usepackage{pict2e}

\if
{

\usepackage[pageref]{backref}
\renewcommand*{\backrefalt}[4]{%
\ifcase #1 %
(Not cited)%
\or
(Cited on p.~#2)%
\else
(Cited on pp.~#2)%
\fi
}

}
\fi

\begin{document}
\title{Convergence and Stability Analysis  of a  Generalized Proximal Point  Algorithm and Its Inexact Version}
\author{Ba Khiet Le \thanks{Analytical and Algebraic Methods in Optimization Research Group, Faculty of Mathematics and Statistics, Ton Duc Thang University, Ho Chi Minh City, Vietnam.
 E-mail: \texttt{lebakhiet@tdtu.edu.vn}. Research of B. K. Le was   supported by National Foundation for Science and
Technology Development (NAFOSTED) of Vietnam under Grant Number
101.01-2025.61}  \qquad Boris S. Mordukhovich\thanks{Department of Mathematics and Center for Artificial Intelligence and Data Science, Wayne State University, Detroit, MI 48202, USA.
 E-mail: \texttt{aa1086@wayne.edu}. Research of B. S. Mordukhovich was partly supported by the US National Science Foundation under grant DMS-2204519 and by the Australian Research Council under Discovery Project DP-190100555} \qquad Michel  A. Th\' era \thanks{XLIM UMR-CNRS 7252 and Universit\'{e} de Limoges,   123 Avenue Albert Thomas,
87060 Limoges CEDEX, France.  E-mail: \texttt{michel.thera@unilim.fr} ORCID 0000-0001-9022-6406}}
\maketitle

\begin{abstract}
In this paper, we introduce the notion of  adaptive strong monotonicity  and establish a linear convergence of the recently proposed  Generalized Proximal Point Algorithm (GPPA)  for computing approximate solutions to nonmonotone inclusion problems in Hilbert spaces. We also investigate the  Inexact Generalized Proximal Point Algorithm (IGPPA)  in the presence of  nondiminishing errors.

\end{abstract}

{\bf Keywords.} Generalized  Proximal Point Algorithm, Inexact Generalized  Proximal Point Algorithm, nondiminishing errors, approximate solutions, R-continuity

{\bf AMS Subject Classification.}  65K05, 65K10, 49J52, 49J53

\section{Introduction and Motivation}

Many optimization problems (see, e.g., \cite{av,abs,BC,br,L1,LT,Mordukhovich,Mordukhovich24,Nesterov3,Rockafellar,roc-wets}) can be formulated as the  inclusion problem
\begin{equation}\label{main}
0 \in F(x),
\end{equation}
where $F:\mathcal{H}\rightrightarrows\mathcal{H}$  is a set-valued operator defined on a Hilbert space $\mathcal H$. When $F$ is maximally monotone, problem \eqref{main} can be solved numerically by the classical {\it Proximal Point Algorithm} \cite{Rockafellar}:
\begin{equation}\label{ppa}
(\textbf{PPA}):\qquad
x_{k+1}=J_{\gamma F}(x_k), \quad
k\ge0, x_0\in\mathcal H, \gamma>0,
\end{equation}
where $J_{\gamma F}:=(Id+\gamma F)^{-1}$ denotes the resolvent of $\gamma F$, and $Id$ is the identity operator. The iteration \eqref{ppa} is equivalent to
\begin{equation}\label{eqppa}
\frac{x_{k+1}-x_k}{\gamma}\in -F(x_{k+1}).
\end{equation}
It is well known (see, e.g., \cite{Rockafellar}) that the sequence $(x_k)$ generated by (\textbf{PPA}) satisfies
\begin{equation}\label{pp0}
\Vert x_{k+1}-x_k\Vert \;\text{ tends to } \; 0\;\mbox{ as }\;k\to\infty
\end{equation}
and converges weakly to a solution of \eqref{main}. Combining \eqref{eqppa} and \eqref{pp0}, we conclude that, for sufficiently large $k$, the iterate $x_k$ is an {\it approximate solution} of \eqref{main}, a property that is particularly desirable in practical computations. On the other hand, weak convergence to an exact solution may be difficult to achieve in the presence of computational errors.

When $F$ is not monotone, a recent extension of the proximal point framework was proposed in \cite{LDT}. The main idea is to construct an associated mapping $v:\mathcal H\to\mathcal H$ such that the pair $(F,v)$ is monotone (see Definition~\ref{mono}) and then to solve \eqref{main} by using the {\it Generalized Proximal Point Algorithm}
\begin{equation}\label{mainal}
(\textbf{GPPA}):\qquad
x_0\in\mathcal H,\qquad
x_{k+1}\in J_{\gamma F}^{v}(x_k),\quad
k=0,1,2,\ldots,
\end{equation}
where the notation
$$
J_{\gamma F}^{v}:=(\gamma F+v)^{-1}\circ v
$$
stands for the {\it warped resolvent} of $\gamma F$ with kernel $v$ \cite{bbc,bc}. The notion of monotonicity of pairs of operators was originally introduced in \cite{acl} for solving a broad class of quasi-variational inequalities and monotone inclusion problems involving single-valued operators. For set-valued operators, the (\textbf{GPPA}) satisfies \cite{LDT}
\begin{equation}\label{eqgppa}
\frac{v(x_{k+1})-v(x_k)}{\gamma}\in -F(x_{k+1}),
\end{equation}
together with the convergence
\begin{equation}\label{ppg0}
\Vert v(x_{k+1})-v(x_k)\Vert\to 0\;\mbox{ as }\;k\to\infty.
\end{equation}
Consequently, for sufficiently large $k$, the iterate $x_k$ is again an approximate solution of \eqref{main}. Moreover, if $v$ is bijective, then the sequence $(x_k)$ converges weakly to a solution of \eqref{main}. If, in addition, the pair $(F,v)$  is strongly monotone, then  $(x_k)$  converges linearly to the unique solution. An inertial version of the (\textbf{GPPA}) has recently been developed in \cite{LMT}, where an improved numerical performance was also demonstrated. {Pair monotonicity has appeared in the literature under various names; see, for instance, \cite{Verma} and \cite{VP}. However, to establish the convergence of their algorithms, these works typically require \(v\) to be bijective (see, e.g., \cite[Theorem 3.3]{Verma} and \cite[Theorem 4]{VP}). In contrast, the approach developed in \cite{LDT} does not rely on this assumption in order to obtain approximate solutions. We also note that, when \(v\) is bijective, a change of variables allows pair monotonicity to be reduced to classical monotonicity; see, for instance, \cite{acl}.}

The {\it first motivation} of this paper is to establish {\it linear convergence}  of the (\textbf{GPPA}) under significantly weaker assumptions. More precisely, we seek to obtain linear convergence without assuming that $v$ is bijective or that the pair $(F,v)$  is strongly monotone. To this end, we introduce a new concept, called \textit{adaptive strong monotonicity} (Definition~\ref{stmono}), which is strictly weaker than the classical strong monotonicity of operator pairs {(see, e.g., \cite{LDT,Verma,VP})}. Under this  assumption, we establish linear convergence of the residual sequence \eqref{ppg0} (Theorem~\ref{lin}), {which  can be used to  consider the (\textbf{IGPPA}) (Theorem~\ref{sIGPPA})}.

Our {\it second goal} is to study {\it robustness} of the newly proposed {\it Inexact Generalized Proximal Point Algorithm} designed by
\begin{equation}\label{igg}
(\textbf{IGPPA}):\qquad
x_0\in\mathcal H,\qquad
x_{k+1}\in J_{\gamma F}^{v}(x_k)+y_{k+1},
\quad
k=0,1,2,\ldots,
\end{equation}
where the {\it computational errors} satisfy the estimate
\begin{equation}\label{error}
\|y_k\|\le\delta,\qquad k=1,2,\ldots,
\end{equation}
for some prescribed constant $\delta>0$. This setting reflects the practical situation in which the warped resolvent $J_{\gamma F}^{v}(x_k)$ is computed only approximately.

%
{Most existing convergence results for inexact algorithms (see, e.g., \cite{Buscaglia,Reich1,Reich,Rockafellar}) require the error sequence to be {summable}, that is,
\[
\sum_{k=1}^{\infty}|y_k|<\infty,
\]
which necessarily implies that $y_k\to0$ and is often unrealistic in practical implementations. Other approaches rely on relative error criteria (see, e.g., \cite{Buscaglia,KLMT,KMT}), whose verification at every iteration may also be computationally demanding.}

{To the best of our knowledge, the first work allowing nondiminishing but uniformly bounded errors under assumption \eqref{error} is due to  Zaslavski \cite{zas10}; see also \cite{zas20}. However, the number of iterations required to obtain an approximate solution may be prohibitively large. More recently, in \cite{LMT1}, the authors studied practical variants of the Inexact Proximal Point Algorithm and the Inexact Tseng Algorithm for maximally monotone operators using Tikhonov regularization. Building upon the notion of adaptive strong monotonicity introduced in that work, we extend this line of research to the \textit{nonmonotone} setting by regularizing problem \eqref{main} and establishing the stability of the (\textbf{IGPPA}); see Theorem~\ref{sIGPPA}.
}

The remainder of the paper is organized as follows. Section~\ref{s2} reviews the necessary preliminaries on pair monotonicity and $R$-continuity. Section~\ref{s3} establishes a linear convergence of the (\textbf{GPPA}) for computing approximate solutions under adaptive strong monotonicity. Section~\ref{s4} is devoted to the convergence and stability analysis of the (\textbf{IGPPA}). Finally, concluding remarks are presented in Section~\ref{s5}.

\section{Preliminaries}\label{s2}

Let $\mathcal{H}$ denotes a real Hilbert space endowed with the inner product $\langle \cdot,\cdot\rangle$ and the associated norm $\|\cdot\|$. Given a set-valued mapping $F:\mathcal{H}\rightrightarrows \mathcal{H}$, the {\it domain}, {\it range}, {\it graph}, and set of {\it zeros} of $F$ are defined respectively by
\begin{align*}
\dom F: &= \{x\in\mathcal{H}\;|\;F(x)\neq\varnothing\},\\
\ran F: &= \bigcup_{x\in\mathcal{H}}F(x),\\
\gra F: &= \{(x,y)\in\mathcal{H}\times\mathcal{H}\;|\;y\in F(x)\},\\
\zer F: &= \{x\in\mathcal{H}\;|\;0\in F(x)\}.
\end{align*}
The {\it inverse} operator of $F$ is defined by
\[
F^{-1}(y)=\{x\in\mathcal{H}\;|\;y\in F(x)\}.
\]
We can see  that
\[
\dom F^{-1}=\ran F,
\qquad
\ran F^{-1}=\dom F.
\]
We say that $F$ is  {\it monotone} if 
\[
\langle x^*-y^*,x-y\rangle\geq 0\;\mbox{ for all }\;(x,x^*),(y,y^*)\in\gra F,
\]
or we can write
\beq
\langle F(x)-F(y),x-y\rangle\geq 0\;\mbox{ for all }\;x, y \in \mathcal{H}.
\eeq
It is said that \(F\) is  {\it maximally monotone}
if \(F\) is monotone and  there is no monotone operator
$G :\mathcal{H}\rightrightarrows \mathcal{H} $ such that
\[
\operatorname{gra} F \subsetneq \operatorname{gra} G.
\]
Furthermore, $F$ is called {\it $\alpha$-strongly monotone}  ($\alpha>0$) if
\[
\langle x^*-y^*,x-y\rangle
\geq \alpha\|x-y\|^2 \mbox{ for all }\;(x,x^*),(y,y^*)\in\gra F,
\]
whenever $(x,x^*),(y,y^*)\in\gra F$. The {\it resolvent} of $F$ is defined  by
\[
J_F:=(I+F)^{-1}.
\]

It is classical that if $F$ is maximally monotone, then $J_F$ is single-valued, everywhere defined on $\mathcal{H}$, and firmly nonexpansive (see, e.g., 
\cite{BC,Rockafellar}). Next we talk about the monotonicity of pairs of operators, introduced in \cite{acl} and developed in \cite{LDT,LMT}.

\begin{definition}\label{mono}
Let be given $F, G: \mathcal{H}\rightrightarrows \mathcal{H}$. The {\sc  pair} $(F,G)$ is called {\sc monotone} if 
$$
\langle F(x)-F(y), G(x)-G(y) \rangle\ge 0\;\mbox{ for all }\;x, y \in \mathcal{H},
$$
{i.e., { for all }$ (x,x_F^*),(y,y_F^*)\in\gra F$ and $ (x,x_G^*),(y,y_G^*)\in\gra G$, one has 
$$
\langle x_F^*-y_F^*, x_G^*-y_G^*\rangle\ge 0. 
$$}
\end{definition}

The notion of monotonicity of pairs of operators is remarkably general, as illustrated by the following example.

\begin{example}
{Let $F:\mathbb{R}^3\rightrightarrows\mathbb{R}^3$ and $v:\mathbb{R}^3\to\mathbb{R}^3$ be defined by
\[
F(x) = \begin{bmatrix}
\sgn(x_1+x_2-x_3) \\
\sgn(x_3) \\
\sgn(x_2) 
\end{bmatrix}+\begin{bmatrix}
0 & 0 & 0\\
4 & 5 & 6\\
7 & 8 & 9 
\end{bmatrix}
\begin{bmatrix}
x_1 \\
x_2\\
x_3
\end{bmatrix}
\quad \text{and} \quad
v(x) = \begin{bmatrix}
1 & 1 & -1\\
0 & 0 & 0\\
0 & 0 & 0 
\end{bmatrix}
\begin{bmatrix}
x_1 \\
x_2\\
x_3
\end{bmatrix}
\]
with the notations
$$
x=\begin{bmatrix}
x_1 \\
x_2\\
x_3
\end{bmatrix}\;\;{\rm and}\;\;
\sgn(a) = \begin{cases}
1 & \text{if } a > 0, \\
[-1, 1] & \text{if } a = 0, \\
-1 & \text{if } a < 0.
\end{cases}
$$
Then the pair $(F,v)$ is monotone, although $F$ itself is not monotone.}
\end{example}

\begin{example}
Suppose that  $F: \mathcal{H} \to \mathcal{H} $ can be decomposed by $F(x)=f(x)+g(x)$ such that
$$
\Vert g(x)-g(y)\Vert\le \Vert f(x)-f(y)\Vert\;\;{\rm for\;all}\;x, y\in \mathcal{H}
$$
and that $f^{-1}$ can be computed easier than $F^{-1}$.
Denoting $v:=f-g$, we have that the pair $(F,v)$ is monotone. 
\end{example}

\begin{definition}\label{stmono-alpha}
Let $F: \mathcal{H}\rightrightarrows \mathcal{H}, \;v: \mathcal{H}\to \mathcal{H}$ be given. The pair $(F,v)$ is called {\sc $\alpha$-strongly monotone} {\rm ($\alpha>0$)} if 
$$
\langle F(x)-F(y), v(x)-v(y) \rangle\ge \alpha \Vert x-y \Vert^2.
$$
\end{definition}

\begin{remark}
The strong monotonicity of the pair $(F,v)$ is a quite restrictive assumption, which usually requires that $v$ is bijective. It motivates us to introduce a new condition, which is weaker while still serving our purpose; namely, to have some useful information about the convergence rate of optimization algorithms. 
\end{remark}

\begin{definition}\label{stmono}
Let the operators $F: \mathcal{H}\rightrightarrows \mathcal{H}$ and $v: \mathcal{H}\to \mathcal{H}$ be given. The pair $(F,v)$ is called {\sc adaptively strongly monotone} if there exists $\alpha>0$ such that
$$
\langle F(x)-F(y), v(x)-v(y) \rangle\ge \alpha \Vert v(x)-v(y)\Vert^2.
$$
\end{definition}

\begin{remark}
Note that if $v$ is Lipschitz continuous and $(F,v)$ is strongly monotone, then $(F,v)$ is adaptively strongly monotone.
\end{remark}


Next we recall the notion of warped resolvents of an operator that is an extension of the classical resolvent corresponding to $v=Id$.

\begin{definition}\label{wa}\cite{bbc,bc,LDT}
Let \( F: \mathcal{H} \rightrightarrows \mathcal{H} \) and \( v: \mathcal{H} \to \mathcal{H} \) with \( \dom \;v = \mathcal{H} \).  
The {\sc warped resolvent} of \( F \) with kernel \( v \) is the set-valued mapping $J_F^v: \mathcal{H} \rightrightarrows \mathcal{H}$ defined by 
\[
J_F^v: = (F + v)^{-1} \circ v,
\]  
provided that \( \ran v \subseteq \ran(F + v) \).
\end{definition}

Finally let us recall the notions of {\it R-continuity}  of a set-valued mapping $\mathcal{A}:\mathcal{H} \rightrightarrows \mathcal{H}$ at a point ${\bar x}$ in its domain \cite{L1,LMTsv,LT}. 

\begin{definition}\label{rdef}
A set-valued mapping \noindent $\mathcal{A}:\mathcal{H} \rightrightarrows \mathcal{H}$ is called {\sc R-continuous} at ${\bar x}\in\dom\,\mathcal{A}$ if there exist a number $\sigma>0$ and a nondecreasing function $\rho: \mathbb{R}^+\to \mathbb{R}^+$ satisfying $\lim_{r\to 0^+}\rho(r)=\rho(0)=0$ such that we have the inclusion
\begin{equation}\label{rcon}
\mathcal{A}(x) \subset \mathcal{A}({\bar x}  )+\rho(\Vert x-{\bar x}   \Vert)\ball\;\mbox{ for all }\;x\in  \ball({\bar x} ,\sigma)
\end{equation}
with the {\sc continuity modulus function} $\rho$ and {\sc radius} $\sigma$. When $\sigma=\infty$, $\mathcal{A}$ is  called {\sc globally R-continuous} at ${\bar x}$. We further say that: 
\begin{itemize}
\item $\mathcal{A}$ is {\sc $R$-Lipschitz continuous} at ${\bar x} $ with modulus {$L>0$}  if $\rho( r)=Lr$.
\item $\mathcal{A}$ is {\sc $R$-H\"older continuous} at ${\bar x}  $  if $\rho( r)=Lr^\theta$ for some $L>0, \theta>0$.
\end{itemize}
\end{definition}\vspace*{0.05in}

\begin{definition}\label{crdef}
A set-valued mapping \noindent $\mathcal{A}:\mathcal{H} \rightrightarrows \mathcal{H}$ is  called {\sc compactly 
R-continuous} at ${\bar x}\in\dom\,\mathcal{A}$ if for every compact set $K\subset {\mathcal{H}}$, there exists a number $\sigma>0$ such that we can find a nondecreasing function $\rho: \mathbb{R}^+\to \mathbb{R}^+$ satisfying $\lim_{r\to 0^+}\rho(r)=\rho(0)=0$ with
\begin{equation*}\label{rconc}
\mathcal{A}(x)\cap K \subset \mathcal{A}({\bar x}  )+\rho(\Vert x-{\bar x}   \Vert)\ball\;\mbox{ for all }\;x\in  \ball({\bar x} ,\sigma).
\end{equation*}
Similarly to Definition~{\rm\ref{rdef}}, we say that:
\begin{itemize}
\item $\mathcal{A}$ is {\sc compactly $R$-Lipschitz continuous} at ${\bar x} $ with modulus $L>0$ if $\rho( r)=Lr$.

\item $\mathcal{A}$ is {\sc compactly $R$-H\"older continuous} at ${\bar x}$ if $\rho( r)=Lr^\theta$ for some 
$L, \theta>0$.
\end{itemize}
\end{definition}\vspace*{0.05in}

Some characterizations and sufficient conditions for the R-continuity and compact R-continuity notions can be found in  the recent papers \cite{L1,LMTsv,LT}. 

\section{Linear Convergence of \textbf{(GPPA)}}\label{s3}

We begin with recalling a convergence result showing that \textbf{(GPPA)} generates approximate solutions of \eqref{main}.

\begin{theorem}\cite{LDT,LMT}
Suppose that \(S:=\zer F\neq\varnothing\), \((F,v)\) is monotone, and
\[
\ran v\subseteq \ran(\gamma F+v)
\]
for some \(\gamma>0\). Let \((x_k)_{k\in\mathbb N}\) be the sequence generated by Algorithm~\textbf{(GPPA)}. Then the following assertions hold:

\begin{itemize}
\item[(a)] \(\|v(x_{k+1})-v(x_k)\|\to0\) as \(k\to\infty\).

\item[(b)] If \(F\) has a strongly--weakly closed graph, i.e.,
\[
u_k\in F(v_k),\quad
u_k\to u,\quad
v_k\rightharpoonup v
\quad\Longrightarrow\quad
u\in F(v),
\]
then every weak cluster point of \((x_k)\) is a solution of \eqref{main}.

\item[(c)] If \(F^{-1}\) is \(R\)-continuous at \(0\), then
\[
d(x_k,S)\to0
\quad\text{as }k\to\infty.
\]
\end{itemize}
\end{theorem}

We now obtain a new result establishing a  {\it linear convergence rate}  of \textbf{(GPPA)} under the assumption of adaptive strong monotonicity.

\begin{theorem}\label{lin}
Suppose that \(S:=\zer F\neq\varnothing\), \((F,v)\) is
\(\epsilon\)-adaptively strongly monotone, and
\[
\ran v\subseteq\ran(\gamma F+v)
\]
for some \(\epsilon>0\) and \(\gamma>0\). Let
\((x_k)_{k\in\mathbb N}\) be a sequence generated by Algorithm~\textbf{(GPPA)}. Then the following assertions hold:

\begin{itemize}
\item[(a)] For every \(x^*\in S\), we have the estimates
\[
\|v(x_{k+1})-v(x^*)\|^2
\le
\frac{1}{1+2\gamma\epsilon}
\|v(x_k)-v(x^*)\|^2
\]
and
\[
\|v(x_{k+1})-v(x_k)\|^2
\le
\|v(x_k)-v(x^*)\|^2
-(1+2\gamma\epsilon)
\|v(x_{k+1})-v(x^*)\|^2.
\]

\item[(b)] If \(F^{-1}\) is \(R\)-{Lipschitz} continuous at \(0\), then
\[
d(x_k,S)\to0
\]
with a linear convergence rate as \(k\to\infty\).
\end{itemize}
\end{theorem}

\begin{proof}
(a) By the definition of Algorithm~\textbf{(GPPA)}, we have the inclusion
\begin{equation}\label{chagp}
v(x_{k+1})-v(x_k)\in-\gamma F(x_{k+1})
\end{equation}
together with the one
\[
0\in-\gamma F(x^*)
\]
by taking into account that \(x^*\in S\).

Using the \(\epsilon\)-adaptive strong monotonicity of \((F,v)\) tells us that
\[
\langle
v(x_{k+1})-v(x_k),
v(x_{k+1})-v(x^*)
\rangle
\le
-\gamma\epsilon
\|v(x_{k+1})-v(x^*)\|^2.
\]
Equivalently, it says that
\begin{align*}
2(1+\gamma\epsilon)
\|v(x_{k+1})-v(x^*)\|^2
\le
2\langle
v(x_k)-v(x^*),
v(x_{k+1})-v(x^*)
\rangle\\
\le
\|v(x_{k+1})-v(x^*)\|^2
+\|v(x_k)-v(x^*)\|^2
-\|v(x_{k+1})-v(x_k)\|^2,
\end{align*}
where the second inequality follows from the polarization identity. Therefore,
\[
(1+2\gamma\epsilon)
\|v(x_{k+1})-v(x^*)\|^2
\le
\|v(x_k)-v(x^*)\|^2
-\|v(x_{k+1})-v(x_k)\|^2,
\]
which obviously verifiers the desired estimates
\[
\|v(x_{k+1})-v(x^*)\|^2
\le
\frac{1}{1+2\gamma\epsilon}
\|v(x_k)-v(x^*)\|^2
\]
and
\[
\|v(x_{k+1})-v(x_k)\|^2
\le
\|v(x_k)-v(x^*)\|^2
-(1+2\gamma\epsilon)
\|v(x_{k+1})-v(x^*)\|^2.
\]
Consequently, both
\(\|v(x_{k+1})-v(x^*)\|\) and
\(\|v(x_{k+1})-v(x_k)\|\)
converge to zero linearly.\\

(b) Assume that \(F^{-1}\) is \(R\)-Lipschitz continuous at \(0\) with modulus \(L>0\). By \eqref{chagp},
\[
-\frac1\gamma\bigl(v(x_{k+1})-v(x_k)\bigr)\in F(x_{k+1}),
\]
which is equivalently written as
\[
x_{k+1}
\in
F^{-1}\!\left(
-\frac{v(x_{k+1})-v(x_k)}{\gamma}
\right).
\]
By the \(R\)-Lipschitz continuity of \(F^{-1}\) at \(0\), we get
\[
x_{k+1}
\in
F^{-1}(0)
+
\frac{L}{\gamma}
\|v(x_{k+1})-v(x_k)\|
\,\mathbb B
=
S+
\frac{L}{\gamma}
\|v(x_{k+1})-v(x_k)\|
\,\mathbb B,
\]
which readily implies that
\[
d(x_{k+1},S)
\le
\frac{L}{\gamma}
\|v(x_{k+1})-v(x_k)\|.
\]
Since
\(\|v(x_{k+1})-v(x_k)\|\)
converges to zero linearly, so does
\(d(x_{k+1},S)\).
\end{proof}

\begin{remark}
If $(x_k)$ is bounded and $\mathcal{H}=\R^n$, then we can replace the R-Lipschitz continuity of $F^{-1}$ at zero by the compactly R-Lipschitz continuity of $F^{-1}$ at zero; see also \cite{LMTsv,LMT1}.
\end{remark}

\section{Stability and Convergence Analysis of \textbf{(IGPPA)}}\label{s4}
In this section, we study the convergence of \textbf{(IGPPA)} under the {\it nondiminishing errors}. The idea is to regularize the original problem (\ref{main}) by the problem
\begin{equation}\label{remain}
0 \in F(x)+\epsilon v(x),
\end{equation}
where the pair $(F,v)$ is monotone and $\epsilon>0$. 

\begin{lemma}\label{esti}
Suppose that $(F,v)$ is monotone and that $x_1\in \zer F, x_2\in \zer (F+\epsilon v)$ for some $\epsilon>0$. Then we have the estimate
$$
\Vert v(x_2) \Vert\le \Vert v(x_1)\Vert.
$$
\end{lemma}
\begin{proof} 
Observe the inclusions 
$$
0\in F(x_1),\;\;{\rm } -\epsilon v(x_2)\in F(x_2).
$$
Since $(F,v)$ is monotone, we deduce that
$$
\langle \epsilon v(x_2), v(x_1)-v(x_2)\rangle\ge 0,
$$
which implies in turn that
$$
\Vert v(x_2) \Vert^2\le \langle v(x_1),v(x_2)\rangle\le \Vert v(x_1)\Vert \Vert v(x_2)\Vert.
$$
Therefore, it follows that
$$
\Vert v(x_2) \Vert\le \Vert v(x_1)\Vert
$$
as claimed in the lemma.
\end{proof}


\begin{theorem}\label{sIGPPA}
Suppose that $(F,v)$ is monotone, $v$ is $L$-Lipschitz continuous, and
\[
S:=\zer F\neq\varnothing,
\qquad
\zer(F+\epsilon v)\neq\varnothing,
\]
for some  number $\epsilon>0$ sufficiently small. Assume further that $F^{-1}$ is $R$-continuous at $0$ with modulus function $\rho$. Let $(x_k)$ be the sequence generated by the \textbf{(IGPPA)} applied to the regularized problem \eqref{remain}; namely,
\[
x_0\in\mathcal H,\qquad
x_{k+1}\in J_{\gamma\widetilde F}^{\,v}(x_k)+y_{k+1},
\quad k=0,1,2,\ldots,
\]
with the notation
\[
\widetilde F:=F+\epsilon v
\]
and the error sequence satisfying
\[
\|y_{k+1}\|\le\delta,
\]
for some $\delta>0$.\\[0.5ex]
Then by choosing $\delta=\epsilon^2$, we get the estimate
\[
\limsup_{k\to\infty}d(x_k,S)
\le
\epsilon^2+
\rho\!\left(
\frac{L\epsilon^2+4L\epsilon/\gamma}{\gamma}
+\epsilon\left(\frac{2L\epsilon}{\gamma}+a\right)
\right),
\]
where the number $a$ is defined as
\[
a:=\inf_{\tilde x\in S}\|v(\tilde x)\|.
\]
{Consequently, $x_k$ is closed to $S$ if $\epsilon$ is sufficiently small. }
\end{theorem}

\begin{proof}
It follows from the definition
\[
\widetilde F=F+\epsilon v
\]
that the pair $(\widetilde F,v)$ is $\epsilon$-adaptively strongly monotone. Indeed, we have
\begin{align*}
\langle
\widetilde F(x)-\widetilde F(y),
v(x)-v(y)
\rangle
&=
\langle
F(x)-F(y),
v(x)-v(y)
\rangle
+\epsilon\|v(x)-v(y)\|^2\\
&\ge
\epsilon\|v(x)-v(y)\|^2,
\end{align*}
for all $x,y\in\mathcal H$. Define further
\[
{u_k:=x_k-y_k,\qquad k\ge1}
\]
and observe the equivalence
\[
u_{k+1}\in J_{\gamma\widetilde F}^{\,v}(x_k)
\iff
v(u_{k+1})-v(x_k)\in-\gamma\widetilde F(u_{k+1}).
\]

Picking $x^*\in\zer\widetilde F$ and proceeding as in the proof of Theorem~\ref{lin}(a), we obtain
\begin{align*}
\|v(u_{k+1})-v(x^*)\|
&\le
\kappa\|v(x_k)-v(x^*)\|\\
&\le
\kappa\bigl(
\|v(x_k)-v(u_k)\|
+\|v(u_k)-v(x^*)\|
\bigr)\\
&\le
\kappa\bigl(
L\delta+\|v(u_k)-v(x^*)\|
\bigr)
\end{align*}
with the constant $\kappa>0$ defined by
\[
\kappa:=\frac1{\sqrt{1+2\gamma\epsilon}}.
\]
Iterating the latter inequality yields
\begin{align*}
\|v(u_{k+1})-v(x^*)\|
&\le
L\delta\sum_{i=1}^{k+1}\kappa^i
+\kappa^{k+1}\|v(u_0)-v(x^*)\|\\
&\le
\frac{L\delta\kappa}{1-\kappa}
+\kappa^{k+1}\|v(u_0)-v(x^*)\|\\
&\le
\frac{2L\delta}{\gamma\epsilon}
+\kappa^{k+1}\|v(u_0)-v(x^*)\|.
\end{align*}

Choose now $\delta=\epsilon^2$ and define
\[
C(\epsilon,k)
:=
\frac{2L\epsilon}{\gamma}
+\kappa^{k+1}\|v(u_0)-v(x^*)\|,
\]
which brings us to the estimate
\[
\|v(u_{k+1})-v(x^*)\|
\le
C(\epsilon,k),
\]
and
\[
\lim_{k\to\infty}C(\epsilon,k)
=
\frac{2L\epsilon}{\gamma},
\]
It follows from Lemma~\ref{esti} that
\[
\|v(x^*)\|
\le
a,
\qquad
a:=\inf_{\tilde x\in S}\|v(\tilde x)\|,
\]
and therefore we get
\[
\|v(u_{k+1})\|
\le
C(\epsilon,k)+a.
\]
We have
\baqn
\Vert v(x_{k+1} )-v(x^*)\Vert &\le& \Vert v(x_{k+1} )-v(u_{k+1})\Vert +\Vert v(u_{k+1} )-v(x^*)\Vert \\
&\le& L\delta+C(\epsilon,k)=L\epsilon^2+C(\epsilon,k).
\eaqn
Moreover, it follows from the above that
\baqn
\Vert v(u_{k+1} )-v(x_k)\Vert &\le& \Vert v(u_{k+1} )-v(x^*)\Vert + \Vert v(x_{k} )-v(x^*)\Vert \\
&\le& L\epsilon^2+2C(\epsilon,k-1).
\eaqn
Note further that
$$
v(u_{k+1})-v(x_k)\in -\gamma \tilde{F}(u_{k+1})=-\gamma {F}(u_{k+1})-\gamma \epsilon v(u_{k+1})
$$
resulting in the inclusions
\baqn
u_{k+1}&\in& F^{-1}(\frac{ v(u_{k+1})-v(x_k)}{-\gamma}-\epsilon v(u_{k+1}))\\
&\subseteq& F^{-1}(0)+\rho\Big(\frac{\Vert v(u_{k+1})-v(x_k)\Vert}{\gamma}+\epsilon \Vert v(u_{k+1}\Vert\Big)\ball \\
&\subseteq& S +\rho\Big(\frac{ L\epsilon^2+2C(\epsilon,k-1)}{\gamma}+\epsilon (C(\epsilon,k)+a)\Big)\ball.
\eaqn
The latter readily guarantees the distance estimates
$$
d(u_{k+1},S)\le \rho\Big(\frac{ L\epsilon^2+2C(\epsilon,k-1)}{\gamma}+\epsilon (C(\epsilon,k)+a)\Big)\,
$$
and
\baqn
d(x_{k+1},S)\le d(u_{k+1},S)+\Vert x_{k+1}-u_{k+1}\Vert\le   \rho\Big(\frac{ L\epsilon^2+2C(\epsilon,k-1)}{\gamma}+\epsilon (C(\epsilon,k)+a)\Big)+\epsilon^2.
\eaqn
Consequently, we arrive at 
\begin{equation*}
\limsup_{k\to\infty} d(x_k,S)
\leq
\epsilon^2+\rho\Big(\frac{ L\epsilon^2+4L \epsilon\gamma^{-1}}{\gamma}+\epsilon (2L \epsilon\gamma^{-1}+a)\Big)
\end{equation*}
{The conclusion follows since $\rho$ is a modulus function.}
\end{proof}


\begin{remark}
Without loss of generality, it is possible to assume that $L=1$. Indeed, if $v$ is $L$-Lipschitz continuous and the pair $(F,v)$ is monotone, then the pair $(F,\widetilde v)$ with
\[
\widetilde v:=\frac{1}{L}v,
\]
is clearly monotone as well.
\end{remark}

\section{Concluding Remarks and Future Research}\label{s5}

In this paper, we established the linear convergence of \textbf{(GPPA)} for computing approximate solutions of nonmonotone inclusion problems under the newly introduced condition of adaptive strong monotonicity for pairs of operators. The obtained convergence result provides a theoretical foundation for studying the inexact generalized proximal point algorithm \textbf{(IGPPA)}, which allows the presence of nonvanishing computational errors.

Future research will focus on developing further theoretical properties and practical applications of pair monotonicity as well as our technique dealing with inexact algorithms. In particular, it would be of interest to investigate broader classes of optimization and variational problems for which this framework yields efficient and robust numerical algorithms.

\end{document}